[1994/12/01]
\documentclass{ijmart-mod}
\chardef\bslash=`\\ 

\usepackage{graphicx}
\usepackage[breaklinks=true]{hyperref}
\usepackage{mathtools}

\newtheorem{thm}{Theorem}[section]
\newtheorem{cor}[thm]{Corollary}
\newtheorem{lem}[thm]{Lemma}
\newtheorem{prop}[thm]{Proposition}

\theoremstyle{definition}

\newtheorem{rem}[thm]{Remark}

\theoremstyle{remark}

\newcommand{\eval}[2][\right]{\relax
  \ifx#1\right\relax \left.\fi#2#1\rvert}

\begin{document}
\title{Shallow sections of the hypercube}

\author[L. Pournin]{Lionel Pournin}
\address{Universit{\'e} Paris 13, Villetaneuse, France}
\email{lionel.pournin@univ-paris13.fr}

\begin{abstract}
Consider a $d$-dimensional closed ball $B$ whose center coincides with that of the hypercube $[0,1]^d$. Pick the radius of $B$ in such a way that the vertices of the hypercube are outside of $B$ and the midpoints of its edges in the interior of $B$. It is known that, when $d\geq3$, the $(d-1)$-dimensional volume of $H\cap[0,1]^d$, where $H$ is a hyperplane of $\mathbb{R}^d$ tangent to $B$, is largest possible if and only if $H$ is orthogonal to a diagonal of the hypercube. It is shown here that the same holds when $d\geq5$ but the interior of $B$ is only required to contain the centers of the square faces of the hypercube.
\end{abstract}
\maketitle

\section{Introduction}\label{HS.sec.0}

Although many different polytopes can be obtained as the intersection of the hypercube $[0,1]^d$ with a hyperplane of $\mathbb{R}^d$, there exists an elegant general formula for their volume. This formula can be traced back to the work of George P{\'o}lya \cite{Polya1913} (see \cite{Berger2010,FrankRiede2012,KonigKoldobsky2011,Zong2006}, where it is further discussed). It also appears in the proof by Keith Ball \cite{Ball1986} of a conjecture of Douglas Hensley \cite{Hensley1979} on the largest possible volume for the intersection of the hypercube with a hyperplane through its center. This formula, an improper integral of a smooth function (see Theorem~\ref{HS.sec.1.thm.1} below), might look surprising at first because it does not explicitly reflect the non-smooth nature of the hypercube's boundary. 

Given a $d$-dimensional ball $B$ that shares its center with $[0,1]^d$, the problem posed in \cite{Hensley1979} can be generalized by asking for the largest (or smallest) possible $(d-1)$-dimensional volume of $H\cap[0,1]^d$, where $H$ is a hyperplane of $\mathbb{R}^d$ tangent to $B$. Of course, if $B$ contains $[0,1]^d$, the volume is always $0$, and it shall be required that the vertices of $[0,1]^d$ are outside of $B$. This question has first been studied in \cite{MoodyStoneZachZvavitch2013}, where it is shown that, if $d\geq3$ and $B$ contains the midpoints of the edges of $[0,1]^d$ in its interior, then the volume of $H\cap[0,1]^d$ is maximal if and only if $H$ is orthogonal to a diagonal of the hypercube. Here, the geometry of the hypercube plays an important role. Indeed, in this particular case, one of the closed half-spaces of $\mathbb{R}^d$ bounded by $H$ contains at most one vertex of $[0,1]^d$, which allows for a simpler formula that the one from \cite{Ball1986}. In fact, some of the hyperplanes tangent to $B$ might still possess this property even if the midpoints of the edges of $[0,1]^d$ are outside of the interior of $B$. This is the case if and only if the interior of $B$ contains the centers of the simplices spanned by the $d$ vertices of $[0,1]^d$ adjacent to any given vertex. Hermann K{\"o}nig proved that, if $d\geq5$ then for any such ball, the volume of $H\cap[0,1]^d$ is locally maximal when $H$ is orthogonal to a diagonal of the hypercube \cite{Konig2021}. These and similar questions have also been studied in the case of convex bodies other than the hypercube, as for instance the cross-polytope \cite{Konig2021,LiuTkocz2020}, the balls of the $q$-norms  where $2\leq{q}<\infty$ \cite{Koldobsky2005,MeyerPajor1988}, and the regular simplex \cite{Konig2021,Webb1996}.

Here, the result from \cite{MoodyStoneZachZvavitch2013} and, partially, the result from \cite{Konig2021}, are extended by showing that, if $d\geq5$ and $B$ is only required to contain the centers of the square faces of $[0,1]^d$ in its interior, then the volume of $H\cap[0,1]^d$ is maximal if and only if $H$ is orthogonal to a diagonal of the hypercube. 


This result requires to go beyond the case when $H$ separates a single vertex of $[0,1]^d$ from its other vertices. One of the ingredients of the proof is a volume formula for $H\cap[0,1]^d$ that, unlike the one mentioned above, takes the form of a sum over the vertices of $[0,1]^d$ contained in one of the half-spaces of $\mathbb{R}^d$ bounded by $H$ (see Theorem~\ref{HS.sec.1.thm.2} below). This formula is established in Section~\ref{HS.sec.1}. The regularity of the volume of $H\cap[0,1]^d$, as a function of a normal vector of $H$ is studied in the same section, based on the integral volume formula. The case when a unique vertex of $[0,1]^d$ lies in one of the half-spaces bounded by $H$ is treated in Section \ref{HS.sec.2}. The case when two vertices of $[0,1]^d$ lie in that half-space is treated in Section \ref{HS.sec.3} and the main result obtained as a consequence. The possible extension of that result to deeper sections of the hypercube is discussed in Section \ref{HS.sec.4}. Note in particular that, while the developed tools can still be used when $H$ cuts the hypercube in more complicated ways, the calculations require overcoming the exponential combinatorics of the hypercube. 


\section{Volume formulae for sections of the hypercube}\label{HS.sec.1}

The following theorem (see \cite{Ball1986,Berger2010,FrankRiede2012,KonigKoldobsky2011,Polya1913,Zong2006} for equivalent statements) provides the volume of a section of the hypercube by an arbitrary hyperplane. It is stated here for the hypercube $[-1,1]^d$ instead of $[0,1]^d$, as the formula is simpler in this case, but this will not make a difference later on.


\begin{thm}\label{HS.sec.1.thm.1}
Consider the hyperplane $H$ of $\mathbb{R}^d$ made up of the points $x$ satisfying $a\mathord{\cdot}x=b$, where $a$ is a vector from $(\mathbb{R}\mathord{\setminus}\{0\})^d$ and $b$ is a real number. If $d\geq2$, then the $(d-1)$-dimensional volume of $H\cap[-1,1]^d$ is
$$
\frac{2^{d-1}\|a\|}{\pi}\int_{-\infty}^{+\infty}\!\left(\prod_{i=1}^d\frac{\sin(a_iu)}{a_iu}\right)\!\cos(bu)du\mbox{.}
$$
\end{thm}

An equivalent formula has been established by Rolfdieter Frank and Harald Riede \cite{FrankRiede2012} that takes the form of a (finite) sum over the vertices of the hypercube instead of an improper integral. Interestingly, this formula can be simplified into a sum over only the vertices of the hypercube that are contained in one of the half-spaces bounded by the considered hyperplane, as shown in \cite{DezaHiriart-UrrutyPournin2021} for the special case when the hyperplane is orthogonal to the diagonals of the hypercube. It is noteworthy that in this special case, one finds among the resulting intersections a collection of polytopes called the hypersimplices \cite{GabrielovGelfandLosik1975,GelfandGoreskyMacPhersonSerganova1987} whose volume is proportional to the Eulerian numbers \cite{Laplace1886,Stanley1977}. It turns out that the formula from \cite{FrankRiede2012} can be simplified in general as a straightforward consequence of a result by David Barrow and Philip Smith \cite{BarrowSmith1979} that provides the volume of the intersection of the hypercube with one of the half-spaces bounded by the hyperplane. In the sequel, $\sigma(x)$ denotes the sum of the coordinates of a point $x$ from $\mathbb{R}^d$ and $\pi(x)$ denotes their product.

\begin{thm}\label{HS.sec.1.thm.2}
Consider the hyperplane $H$ of $\mathbb{R}^d$ made up of the points $x$ satisfying $a\mathord{\cdot}x=b$, where $a$ is a vector from $(\mathbb{R}\mathord{\setminus}\{0\})^d$ and $b$ is a real number. The $(d-1)$-dimensional volume of $H\cap[0,1]^d$ is
\begin{equation}\label{HS.sec.1.thm.2.eq.0}
\sum\frac{(-1)^{\sigma(v)}\|a\|(b-a\mathord{\cdot}v)^{d-1}}{(d-1)!\pi(a)}\mbox{,}
\end{equation}
where the sum is over the vertices $v$ of $[0,1]^d$ such that $a\mathord{\cdot}v\leq{b}$.
\end{thm}
\begin{proof}
Denote by $H^-$ the half-space of $\mathbb{R}^d$ made up of the points $x$ such that $a\mathord{\cdot}x\leq{b}$. It is proven in \cite{BarrowSmith1979} (see equation (2) therein) that the $d$-dimensional volume of $H^-\cap[0,1]^d$ is given by the expression
\begin{equation}\label{HS.sec.1.thm.2.eq.1}
\sum\frac{(-1)^{\sigma(v)}(b-a\mathord{\cdot}v)^d}{d!\pi(a)}\mbox{,}
\end{equation}
where the sum is over the vertices $v$ of $[0,1]^d$ such that $a\mathord{\cdot}v\leq{b}$. Now observe that the distance $t$ between $H$ and the center of $[0,1]^d$ is
$$
t=\frac{|b-\sigma(a)/2|}{\|a\|}\mbox{.}
$$

Assume for a moment that $b$ is at least $\sigma(a)/2$ or, equivalently, that the center of $[0,1]^d$ is contained in $H^-$. In this case,
$$
b=\frac{\sigma(a)}{2}+t\|a\|\mbox{.}
$$

Replacing $b$ in (\ref{HS.sec.1.thm.2.eq.1}) by the right-hand side of this equality and differentiating with respect to $t$ results in the $(d-1)$-dimensional volume of $H\cap[0,1]^d$. The expression that results from the differentiation is (\ref{HS.sec.1.thm.2.eq.0}), as desired.

Now if $b$ is less than $\sigma(a)/2$, then
$$
b=\frac{\sigma(a)}{2}-t\|a\|\mbox{.}
$$

As above, replacing $b$ in (\ref{HS.sec.1.thm.2.eq.1}) by the right-hand side of this equality and differentiating with respect to $-t$ results in (\ref{HS.sec.1.thm.2.eq.0}).
\end{proof}

\begin{rem}
Proposition 4.1 from \cite{Konig2021} is the special case of Theorem \ref{HS.sec.1.thm.2} when the hyperplane separates one vertex of the hypercube from all the others.
\end{rem}

Recall that we are interested in the hyperplanes of $\mathbb{R}^d$ tangent to a fixed ball $B$ whose center coincides with that of the hypercube $[0,1]^d$. Little mention is made of $B$ in the sequel, and it will be most of the time implicitly represented by its radius $t$. In other words, $t$ is the distance between the hyperplane and the center of $[0,1]^d$.  It will be useful to keep in mind that the interior of $B$ contains the centers of the $k$-dimensional faces of $[0,1]^d$ if and only if $t>\sqrt{d-k}/2$.

For any vector $a$ in $[0,+\infty[^d$, denote
\begin{equation}\label{HS.sec.1.eq.2}
b=\frac{1}{2}\sum_{i=1}^da_i-t\mbox{.}
\end{equation}

Thereafter, this quantity is thought of as a function of $a$ but, for the sake of simplicity it is denoted by $b$ instead of $b(a)$. Further denote by $H$ the hyperplane made up of the points $x$ satisfying $a\mathord{\cdot}x=b$, that is also thought of as depending on $a$. By construction, when $a$ belongs to the unit sphere $\mathbb{S}^{d-1}$ centered at the origin of $\mathbb{R}^d$, the distance between $H$ and the center of $[0,1]^d$ is precisely $t$. In fact, up to the symmetries of $[0,1]^d$, any hyperplane of $\mathbb{R}^d$ at distance $t$ from the center of $[0,1]^d$ coincides with $H$ for some vector $a$ from $\mathbb{S}^{d-1}\cap[0,+\infty[^d$. It follows from Theorem \ref{HS.sec.1.thm.1} that the $(d-1)$-dimensional volume of $H\cap[0,1]^d$ is a smooth function of $a$ on $]0,+\infty[^d$ as soon as $d\geq3$.

\begin{cor}\label{HS.sec.1.cor.1}
If $d\geq3$, then the $(d-1)$-dimensional volume of $H\cap[0,1]^d$ is a continuously differentiable function of $a$ on the open orthant $]0,+\infty[^d$.
\end{cor}
\begin{proof}
Consider the volume formula provided by Theorem \ref{HS.sec.1.thm.1} and observe the $(d-1)$-dimensional volume of $H\cap[0,1]^d$ can be obtained from it by a straightforward change of variables. Further observe that $\|a\|$ is a continuously derivable function of $a$ on $]0,+\infty[^d$. Hence, by the symmetry of this formula with respect to the coordinates of $a$, it is sufficient to show that the partial derivative
\begin{equation}\label{HS.sec.1.cor.1.eq.1}
\frac{\partial}{\partial{a_1}}\int_{-\infty}^{+\infty}\!\left(\prod_{i=1}^d\frac{\sin(a_iu)}{a_iu}\right)\!\cos(tu)du
\end{equation}
exists and is a continuous function of $a$ on $]0,+\infty[^d$. It is a consequence of (\ref{HS.sec.1.eq.2}) that the cosine term is no longer dependent of $a$ after the change of variables (see also Proposition 1 from \cite{KonigKoldobsky2011}). The differentiability is obtained by Leibniz's rule for differentiation under the integral sign whereby
$$
\frac{\partial}{\partial{a_1}}\int_{-\infty}^{+\infty}\!\left(\prod_{i=1}^d\frac{\sin(a_iu)}{a_iu}\right)\!\cos(tu)du=\int_{-\infty}^{+\infty}\frac{\partial}{\partial{a_1}}\!\left(\prod_{i=1}^d\frac{\sin(a_iu)}{a_iu}\right)\!\cos(tu)du\mbox{,}
$$

However, as the integral is improper, Leibniz's rule requires that
$$
\int_{-n}^{+n}\frac{\partial}{\partial{a_1}}\!\left(\prod_{i=1}^d\frac{\sin(a_iu)}{a_iu}\right)\!\cos(tu)du
$$
converges uniformly when $n$ goes to infinity, in a neighborhood of each point $a$ from $]0,+\infty[^d$. The uniform convergence also implies the desired continuity of the partial derivative. Observe that any point contained in $]0,+\infty[^d$ admits, as a neighborhood, the interior of a closed ball contained in the open orthant $]0,+\infty[^d$. Consider such a closed ball $B$ and assume that $a$ belongs to the interior of $B$. Note that, when $u$ is non-zero,
\begin{equation}\label{HS.sec.1.cor.1.eq.2}
\frac{\partial}{\partial{a_1}}\!\left(\prod_{i=1}^d\frac{\sin(a_iu)}{a_iu}\right)\!\cos(tu)=\!\left(\prod_{i=2}^d\frac{\sin(a_iu)}{a_iu}\right)\!\cos(tu)\frac{\partial}{\partial{a_1}}\frac{\sin(a_1u)}{a_1u}\mbox{.}
\end{equation}

By a straightforward calculation,
$$
\frac{\partial}{\partial{a_1}}\frac{\sin(a_1u)}{a_1u}=\frac{\cos(a_1u)}{a_1}-\frac{\sin(a_1u)}{a_1^2u}\mbox{.}
$$

As a consequence,
\begin{equation}\label{HS.sec.1.cor.1.eq.3}
\left|\frac{\partial}{\partial{a_1}}\!\left(\prod_{i=1}^d\frac{\sin(a_iu)}{a_iu}\right)\!\cos(bu)\right|\leq2\min\!\left\{1,\frac{1}{mu^{d-1}}\right\}\!\mbox{,}
\end{equation}
where $m$ is the smallest possible value of the product of the coordinates of a point from $B$. Under the assumption that $d$ is at least $3$, the desired uniform convergence property therefore follows from Cauchy's criterion.
\end{proof}

\begin{rem}
If one looks at the $2$-dimensional situation, Corollary \ref{HS.sec.1.cor.1} might seem surprising. Indeed, when $d=2$ and $0\leq{t}<1/2$, the partial derivatives of the volume of $H\cap[0,1]^2$ are not continuous at the vectors $a$ from $]0,+\infty[^2$ with a coordinate equal to $b$. Geometrically, this corresponds to the case when $H$ contains a vertex of the square $[0,1]^2$. At such vectors, Theorem \ref{HS.sec.1.thm.2} provides two expressions for the partial derivatives of the volume of $H\cap[0,1]^2$. One of these expressions is always negative and the other always positive.
\end{rem}

Another consequence of Theorem \ref{HS.sec.1.thm.1} is that the $(d-1)$-dimensional volume of $H\cap[0,1]^d$ is a continuous function of $a$ almost everywhere on the boundary of the closed positive orthant $[0,+\infty[^d$ when $d\geq3$.

\begin{cor}\label{HS.sec.1.cor.2}
The $(d-1)$-dimensional volume of $H\cap[0,1]^d$ is a continuous function of $a$ at any point contained in the closed orthant $[0,+\infty[^d$ that admits at least two non-zero coordinates.
\end{cor}
\begin{proof}
Consider a point $x$ from $[0,+\infty[^d$ with at least two non-zero coordinates. 
Assume, without loss of generality, that the first two coordinates of $x$ are non-zero. Consider a closed $d$-dimensional ball $B$ centered at $x$, whose interior is non-empty. Choose the radius of $B$ small enough so that the first two coordinates of all the points it contains are non-zero. Observe that
$$
\int_{-n}^n\!\left(\prod_{i=1}^d\frac{\sin(a_iu)}{a_iu}\right)\!\cos(tu)du
$$
is a continuous function of $a$ for any positive integer $n$ as a definite integral of a continuous function of $a$ and $u$ (under the convention that $\sin(0)/0=1$). In fact, that function converges uniformly on $B$ when $n$ goes to infinity. Indeed, for any non-zero real number $u$ and any vector $a$ from $B$,
$$
\left|\left(\prod_{i=1}^d\frac{\sin(a_iu)}{a_iu}\right)\!\cos(tu)\right|\leq\!\min\left\{1,\frac{1}{mu^2}\right\}\!\mbox{,}
$$
where $m$ is the smallest value for the product of the first two coordinates of a point from $B$. Hence, by the Cauchy criterion, the desired uniform convergence property holds. As a consequence,
$$
\frac{\|a\|}{\pi}\int_{-\infty}^\infty\!\left(\prod_{i=1}^d\frac{\sin(a_iu)}{a_iu}\right)\!\cos(tu)du
$$
is a continuous function of $a$ on $B$. It remains to remark that, according to Theorem \ref{HS.sec.1.thm.2}, the value of this function at any point $a$ from $B$ (still under the convention that $\sin(0)/0=1$, and by the same change of variables as in the proof of Corollary \ref{HS.sec.1.cor.1}) is precisely the volume of $H\cap[0,1]^d$.
\end{proof}

\begin{rem}
While Corollary \ref{HS.sec.1.cor.2} is valid in $2$ dimensions, it merely states in this case that the length of the segment $H\cap[0,1]^2$ is a continuous function of $a$ in the open quarter plane $]0,+\infty[^2$. In fact, that length is not always continuous on the boundary of the closed quarter plane. For instance, if $t=1/2$, it can be easily seen that $H\cap[0,1]^2$ has length $1$ when $a$ has a unique non-zero coordinate and length at most $1/\sqrt{2}$ when both coordinates of $a$ are positive. This corresponds to the limit case when the hyperplane $H$ is tangent to the ball $B$ inscribed in the hypercube $[0,1]^d$. That limit case is particularly interesting: as shown in \cite{KonigKoldobsky2011}, in this case the volume of $H\cap[0,1]^d$ is minimal when $H$ is orthogonal to a diagonal of $[0,1]^d$, in dimensions $2$ and $3$. It is asked in \cite{Konig2021} whether the same holds in higher dimensions. These observations suggest that, as $B$ shrinks from a ball that contains the vertices of the hypercube to the ball inscribed in it, the $(d-1)$-dimensional volume of the intersection between the hypercube and the hyperplanes tangent to $B$ exhibits dramatically different behaviors.
\end{rem}

\begin{prop}\label{HS.sec.1.prop.1}
If $t>1/2$ then the $(d-1)$-dimensional volume of $H\cap[0,1]^d$, as a function of $a$ defined on $\mathbb{S}^{d-1}\cap[0,+\infty[^d$ admits a maximum. 
\end{prop}
\begin{proof}
Assume that $t>1/2$. By that assumption, $H$ is disjoint from $[0,1]^d$ when $a$ has a unique non-zero coordinate. Now observe that $H$ remains disjoint from $[0,1]^d$ in a neighborhood of any such vector $a$. As a consequence, the $(d-1)$-dimensional volume of $H\cap[0,1]^d$ is equal to $0$ in neighborhood of $a$. This, together with Corollary \ref{HS.sec.1.cor.2}, proves that the $(d-1)$-dimensional volume of $H\cap[0,1]^d$ is continuous on the whole of $\mathbb{S}^{d-1}\cap[0,+\infty[^d$. The result therefore  follows from the compactness of $\mathbb{S}^{d-1}\cap[0,+\infty[^d$.
\end{proof}

\section{Cutting the corners of the hypercube}\label{HS.sec.2}

Recall that $a$ is a non-zero vector from $[0,+\infty[^d$, $b$ the affine function of $a$ defined by (\ref{HS.sec.1.eq.2}) and $H$ the hyperplane of $\mathbb{R}^d$ made up of the points $x$ such that $a\mathord{\cdot}x=b$. From now on, the $(d-1)$-dimensional volume of $H\cap[0,1]^d$ will be denoted by $V$ and, just as $b$ and $H$, it is treated as a function of $a$. For any number $\lambda$, consider the Lagrangian function
\begin{equation}\label{HS.sec.2.eq.1}
L_\lambda=\frac{V}{\|a\|}+\lambda\!\left(\|a\|^2-1\right)
\end{equation}

It is an immediate consequence of Corollary \ref{HS.sec.1.cor.1} that $L_\lambda$ is a continuously differentiable function of $a$ on the open orthant $]0,+\infty[^d$. Recall that a critical point of $L_\lambda$ is a vector $a$ in $]0,+\infty[^d$ such that, at $a$,
\begin{equation}\label{HS.sec.2.eq.2}
\frac{\partial{L_\lambda}}{\partial{a_i}}=0
\end{equation}
for every integer $i$ satisfying $1\leq{i}\leq{d}$,

\begin{lem}\label{HS.sec.2.lem.1}
Consider a vector $a$ from $\mathbb{S}^{d-1}\cap[0,+\infty[^d$ and assume that the maximum of $V$ on $\mathbb{S}^{d-1}\cap[0,+\infty[^d$ is attained at $a$. If in addition,
\begin{equation}\label{HS.sec.2.lem.1.eq.1}
\frac{\sqrt{d-1}}{2}<t<\frac{\sqrt{d}}{2}\mbox{,}
\end{equation}
then $a$ belongs to the open orthant $]0,+\infty[^d$, $b$ is positive at $a$, and there exists a number $\lambda$ such that $a$ is a critical point of $L_\lambda$.
\end{lem}
\begin{proof}
Under the assumption that (\ref{HS.sec.2.lem.1.eq.1}) holds, $V$ must be positive at $a$. Indeed, according to (\ref{HS.sec.2.lem.1.eq.1}), the distance between the center $c$ of $[0,1]^d$ and $H$ is less than the distance between $c$ and a vertex of $[0,1]^d$ and, therefore, $V$ is sometimes positive (as, for instance when all the coordinates of $a$ coincide). Since $V$ attains its maximum on $\mathbb{S}^{d-1}\cap[0,+\infty[^d$ at $a$, it is necessarily positive at $a$. 

Now observe that (\ref{HS.sec.2.lem.1.eq.1}) also implies that the distance between $c$ and $H$ is greater than the distance between $c$ and the midpoint of the edges of $[0,1]^d$. Hence, if a coordinate of $a$ were equal to $0$, then $H$ and $[0,1]^d$ would be disjoint and $V$ would be equal to $0$. Therefore $a$ necessarily belongs to $]0,+\infty[^d$. In turn, $b$ must be positive at $a$. Indeed, otherwise $H$ and $[0,1]^d$ would either be disjoint (in the case when $b<0$), or intersect precisely at the origin of $\mathbb{R}^d$ (in the case when $b=0$). In both cases, $V$ would be equal to $0$ at $a$.

Finally, observe that the gradient of
$$
a\rightarrow\|a\|^2-1
$$
is non-zero in $\mathbb{R}^d\mathord{\setminus}\{0\}$. Hence, as $a$ belongs to $]0,+\infty[^d$, it must be a critical point of $L_\lambda$ for some real number $\lambda$ by the Lagrange multipliers theorem.
\end{proof}

\begin{lem}\label{HS.sec.2.lem.2}
Consider a critical point $a$ of $L_\lambda$ contained in $\mathbb{S}^{d-1}\cap]0,+\infty[^d$. Assume that $t<\sqrt{d}/2$ and that either 
\begin{enumerate}
\item[(i)]  $3\leq{d}\leq4$ and $t>\sqrt{d-1}/2$, or
\item[(ii)] $d\geq5$ and $t>\sqrt{d}/2-1/\sqrt{d}$.
\end{enumerate}

If $b$ is positive at $a$ and all the coordinates of $a$ are greater than or equal to $b$, then all of these coordinates necessarily coincide.
\end{lem}

\begin{proof}
Assume that the coordinates of $a$ are all at least $b$ and that $b$ is positive. In this case, the only vertex $v$ of $[0,1]^d$ that satisfies $a\mathord{\cdot}v<b$ is the origin of $\mathbb{R}^d$. It therefore follows from Theorem \ref{HS.sec.1.thm.2} that
$$
V=\frac{b^{d-1}}{(d-1)!\pi(a)}
$$

In turn, according to (\ref{HS.sec.1.eq.2}) and (\ref{HS.sec.2.eq.1}),
\begin{equation}\label{HS.sec.2.lem.2.eq.1}
\frac{\partial{L_\lambda}}{\partial{a_i}}=\frac{1}{(d-1)!\pi(a)}\left(-\frac{b^{d-1}}{a_i}+\frac{d-1}{2}b^{d-2}\right)+2a_i\lambda\mbox{.}
\end{equation}

For any two integers $j$ and $k$ such that $1\leq{j}<k\leq{d}$, (\ref{HS.sec.2.eq.2}) implies that
$$
a_k\frac{\partial{L_\lambda}}{\partial{a_j}}-a_j\frac{\partial{L_\lambda}}{\partial{a_k}}=0\mbox{.}
$$

Using the expression for the partial derivatives from (\ref{HS.sec.2.lem.2.eq.1}) and because of the assumption that $b$ is positive, this can be rewritten as
\begin{equation}\label{HS.sec.2.lem.2.eq.2}
-b\left(\frac{a_k}{a_j}-\frac{a_j}{a_k}\right)+\frac{d-1}{2}(a_k-a_j)=0
\end{equation}

Assume, for contradiction, that $a_j$ and $a_k$ do not coincide. In this case, one can divide (\ref{HS.sec.2.lem.2.eq.2}) by $a_k-a_j$, which yields
\begin{equation}\label{HS.sec.2.lem.2.eq.3}
\frac{b}{a_j}+\frac{b}{a_k}=\frac{d-1}{2}\mbox{.}
\end{equation}

Since both $a_j$ and $a_k$ are at least $b$, this equality only possibly holds when $d$ is at most $5$. If $d$ is equal to $5$, it implies that both $a_j$ and $a_k$ coincide with $b$, contradicting the assumption that $a_j$ and $a_k$ are distinct. If $3\leq{d}\leq4$, then (\ref{HS.sec.2.lem.2.eq.3}) implies that $a_j/2$ and $a_k/2$ cannot both be greater than $b$. Geometrically, this means that one of the closed half-spaces of $\mathbb{R}^d$ bounded by $H$ contains both the origin of $\mathbb{R}^d$ and the midpoint of an edge of $[0,1]^d$ which, in this case, is forbidden by the assumption that $t>\sqrt{d-1}/2$.
\end{proof}

One recovers the result of \cite{MoodyStoneZachZvavitch2013} from Lemma \ref{HS.sec.2.lem.1} and \ref{HS.sec.2.lem.2}.

\begin{thm}\label{HS.sec.2.thm.1}
Assume that $t$ satisfies
\begin{equation}\label{HS.sec.2.thm.1.eq.0}
\frac{\sqrt{d-1}}{2}<t<\frac{\sqrt{d}}{2}
\end{equation}
and consider a hyperplane $H$ at distance $t$ from the center of $[0,1]^d$. If $d\geq3$, then the $(d-1)$-dimensional volume of $H\cap[0,1]^d$ is at most
\begin{equation}\label{HS.sec.2.thm.1.eq.0.5}
\frac{d^{d/2}}{(d-1)!}\!\left(\frac{\sqrt{d}}{2}-t\right)^{\!\!d-1}
\end{equation}
with equality if and only if $H$ is orthogonal to a diagonal of $[0,1]^d$.
\end{thm}
\begin{proof}
Assume that $d$ is at least $3$. It follows from this and from (\ref{HS.sec.2.thm.1.eq.0}) that $t>1/2$. In turn, by Proposition \ref{HS.sec.1.prop.1}, $V$ admits a maximum on $\mathbb{S}^{d-1}\cap[0,+\infty[^d$. Thanks to the symmetries of the hypercube, it suffices to show that the maximum of $V$ is uniquely attained at the vector in $\mathbb{S}^{d-1}\cap]0,+\infty[^d$ whose all coordinates coincide. Indeed, the coordinates of that vector are $1/\sqrt{d}$ and it then follows from Theorem \ref{HS.sec.1.thm.2} and from (\ref{HS.sec.1.eq.2}) that $V$ is equal to the announced volume. As (\ref{HS.sec.2.thm.1.eq.0}) holds, one just needs to combine Lemmas \ref{HS.sec.2.lem.1} and \ref{HS.sec.2.lem.2} in order to prove that the coordinates of any vector from $\mathbb{S}^{d-1}\cap]0,+\infty[^d$ where $V$ attains its maximum on $\mathbb{S}^{d-1}\cap]0,+\infty[^d$ necessarily coincide.
\end{proof}

It will be useful to know how (\ref{HS.sec.2.thm.1.eq.0.5}) behaves as a function of $d$.

\begin{prop}\label{HS.sec.2.prop.1}
Assume that $t$ satisfies
$$
\frac{\sqrt{d-2}}{2}\leq{t}\leq\frac{\sqrt{d-1}}{2}\mbox{.}
$$

If in addition, $d\geq5$, then
$$
\frac{d^{d/2}}{(d-1)!}\!\left(\frac{\sqrt{d}}{2}-t\right)^{\!\!d-1}\!\!>\,\frac{(d-1)^{(d-1)/2}}{(d-2)!}\!\left(\frac{\sqrt{d-1}}{2}-t\right)^{\!\!d-2}\!\!\!\!\!\!\!\!\!\mbox{.}
$$
\end{prop}
\begin{proof}
Observe that the desired inequality can be rewritten as
\begin{equation}\label{HS.sec.2.prop.1.eq.1}
\frac{d^{d/2}}{(d-1)^{(d+1)/2}}>2\frac{(\sqrt{d-1}-2t)^{d-2}}{(\sqrt{d}-2t)^{d-1}}\mbox{.}
\end{equation}

When $d\geq5$, the right-hand side of this inequality is a decreasing function of $t$ on the considered interval. Indeed, its derivative with respect to $t$ is
$$
4A\frac{(\sqrt{d-1}-2t)^{d-3}}{(\sqrt{d}-2t)^{d}}
$$
where $A=(d-1)\sqrt{d-1}-(d-2)\sqrt{d}-2t$. In particular, that derivative has the sign of $A$ on the considered interval for $t$. When $t$ is in that interval,
$$
A\leq(d-1)\sqrt{d-1}-(d-2)\sqrt{d}-\sqrt{d-2}\mbox{.}
$$

Rearranging the terms of the right-hand side, and simultaneously multiplying and dividing it by $(\sqrt{d}+\sqrt{d-1})(\sqrt{d-1}+\sqrt{d-2})$ yields
$$
A\leq\frac{\sqrt{d}(1-\sqrt{d}\sqrt{d-2})+2\sqrt{d-2}-(d-3)\sqrt{d-1}}{(\sqrt{d}+\sqrt{d-1})(\sqrt{d-1}+\sqrt{d-2})}
$$
which, as $\sqrt{d-2}<\sqrt{d-1}$ implies the negativity of $A$ when $d\geq5$. Hence, it suffices to show that the desired inequality holds for the smallest $t$ in the considered interval. For this reason, it is assumed that $t=\sqrt{d-2}/2$ in the remainder of the proof. Observe that when $d=5$, the left-hand side of (\ref{HS.sec.2.prop.1.eq.1}) is then greater than $0.7$ and its right-hand side less than $0.7$, proving the inequality when $d=5$. Assume that $d\geq6$, and observe that
$$
\begin{array}{rcl}
\displaystyle\frac{(\sqrt{d-1}-\sqrt{d-2})^{d-2}}{(\sqrt{d}-\sqrt{d-2})^{d-1}} & = & \displaystyle\frac{1}{2^{d-1}}\frac{(\sqrt{d}+\sqrt{d-2})^{d-1}}{(\sqrt{d-1}+\sqrt{d-2})^{d-2}}\\[\bigskipamount]
 & \leq & \displaystyle\frac{1}{2^{d-2}}\frac{d^{(d-1)/2}}{(d-2)^{(d-2)/2}}\\[\medskipamount]
 & \leq & \displaystyle\frac{1}{\sqrt{d-1}}\!\left(\frac{1}{2}\sqrt{\frac{d}{d-2}}\right)^{\!\!d}\!\!(4d-8)\mbox{.}\\
\end{array}
$$

As a consequence, it is sufficient to show that
$$
\left(2\sqrt{\frac{d-2}{d-1}}\right)^{\!\!d}\!\!>8d-16\mbox{.}
$$

Now observe that, since $d\geq6$,
$$
2\sqrt{\frac{d-2}{d-1}}\geq\frac{4}{\sqrt{5}}\mbox{.}
$$

Finally, as $(4/\sqrt{5})^d-8d+16$ is positive when $d=6$ and an increasing function of $d$ on $[6,+\infty[$ (because its derivative is positive), (\ref{HS.sec.2.prop.1.eq.1}) holds.
\end{proof}

Theorem \ref{HS.sec.2.thm.1} allows to extend Lemma \ref{HS.sec.2.lem.1} as follows.

\begin{lem}\label{HS.sec.2.lem.3}
Consider a vector $a$ from $\mathbb{S}^{d-1}\cap[0,+\infty[^d$ and assume that the maximum of $V$ on $\mathbb{S}^{d-1}\cap[0,+\infty[^d$ is attained at $a$. If in addition,
\begin{equation}\label{HS.sec.2.lem.3.eq.1}
\frac{\sqrt{d-2}}{2}<t<\frac{\sqrt{d}}{2}\mbox{,}
\end{equation}
then $a$ belongs to the open orthant $]0,+\infty[^d$, $b$ is positive at $a$, and there exists a number $\lambda$ such that $a$ is a critical point of $L_\lambda$.
\end{lem}
\begin{proof}
Recall that, if all the coordinates of $a$ are equal, then $V$ is
\begin{equation}\label{HS.sec.2.lem.3.eq.2}
\frac{d^{d/2}}{(d-1)!}\!\left(\frac{\sqrt{d}}{2}-t\right)^{\!\!d-1}\!\!\!\!\!\!\!\!\!\mbox{.}
\end{equation}

As $V$ attains its maximum on $\mathbb{S}^{d-1}\cap[0,+\infty[^d$ at $a$, then $V$ must be at least (\ref{HS.sec.2.lem.3.eq.2}). According to (\ref{HS.sec.2.lem.3.eq.1}), the distance between the center of $[0,1]^d$ and $H$ is greater than the distance between the center of $[0,1]^d$ and the center of a square face of $[0,1]^d$. In particular, if two of the coordinates of $a$ were equal to $0$, then $H$ would be disjoint from $[0,1]^d$. Since $V$ is positive, this shows that all the coordinates of $a$ are non-zero except possibly one.

Assume for contradiction that a unique coordinate of $a$ is equal to zero and identify for a moment $\mathbb{R}^{d-1}$ with the subspace of $\mathbb{R}^d$ spanned by the other $d-1$ coordinates. In this case, $V$ is equal to the $(d-2)$-dimensional volume of $H\cap[0,1]^{d-1}$. Now observe that (\ref{HS.sec.2.lem.3.eq.1}) can be rewritten as
$$
\frac{\sqrt{(d-1)-1}}{2}<t<\frac{\sqrt{(d-1)+1}}{2}\mbox{.}
$$

If $t\geq\sqrt{d-1}/2$, then either $H$ and $[0,1]^{d-1}$ are disjoint or their intersection is a vertex of $[0,1]^{d-1}$. In both of these cases, $V$ would be equal to zero, which is impossible. If $t<\sqrt{d-1}/2$,  then by Theorem \ref{HS.sec.2.thm.1}, $V$ is at most
$$
\frac{(d-1)^{(d-1)/2}}{(d-2)!}\!\left(\frac{\sqrt{d-1}}{2}-t\right)^{d-2}
$$
and according to Proposition \ref{HS.sec.2.prop.1}, this quantity is less than (\ref{HS.sec.2.lem.3.eq.2}), contradicting the above observation that $V$ is at least (\ref{HS.sec.2.lem.3.eq.2}). This proves that $a$ belongs to $]0,+\infty[^d$. One obtains that $b$ is positive at $a$ and that, for some number $\lambda$, $a$ is a critical point of $L_\lambda$ by the same argument as in the proof of Lemma \ref{HS.sec.2.lem.1}.
\end{proof}

\section{Cutting the edges of the hypercube}\label{HS.sec.3}

This section treats the case when the hypercube $[0,1]^d$ admits exactly two of its vertices on one side of $H$. In other words, instead of just cutting a corner of the hypercube, $H$ cuts an edge. As a first step, Lemma \ref{HS.sec.2.lem.2} will be generalized to this case. The following proposition will be useful. Its proof is only technical, and it is postoned until the end of the section.

\begin{prop}\label{HS.sec.3.prop.1}
Consider a number $y$ in $]0,1[$ and denote
$$
\left\{
\begin{array}{l}
\alpha=2-2y^{d-1}-(d-1)(1-y)(1+y^{d-2})\mbox{,}\\
\beta=(d-1)(1-y)^2(1-y^{d-2})\mbox{,}\\
\gamma=-2(1-y)^2(1-y^{d-1})\mbox{.}\\
\end{array}
\right.
$$

If $d\geq6$, then the quadratic equation
$$
\alpha{x^2}+\beta{x}+\gamma=0
$$
does not admit a solution in the interval $[1,+\infty[$.
\end{prop}

The announced generalization of Lemma \ref{HS.sec.2.lem.2} can now be given.

\begin{lem}\label{HS.sec.3.lem.1}
Consider a critical point $a$ of $L_\lambda$ contained in $\mathbb{S}^{d-1}\cap]0,+\infty[^d$. Assume that $t<\sqrt{d}/2$ and that either 
\begin{enumerate}
\item[(i)] $d=5$ and $t>\sqrt{d-2}/2$, or
\item[(ii)] $d\geq6$ and $t>\sqrt{d}/2-1/\sqrt{d}$.
\end{enumerate}

If $b$ is positive at $a$ and at least $d-1$ of the coordinates of $a$ are greater than or equal to $b$, then all of the coordinates of $a$ necessarily coincide.
\end{lem}

\begin{proof}
Assume that $b$ is positive. It will be shown that, under the assumptions made on $d$ and $t$, none of the coordinates of $a$ can be less than $b$. Observe that, in this case, the result then follows from Lemma \ref{HS.sec.2.lem.2}. Assume, for contradiction that $a$ has a unique coordinate less than $b$. It can be required without loss of generality that this coordinate is the first. In this case, the only two vertices $v$ of $[0,1]^d$ such that $a\mathord{\cdot}v<b$ are the origin of $\mathbb{R}^d$ and the vertex whose only non-zero coordinate is the first. Hence, by Theorem \ref{HS.sec.1.thm.2},
$$
V=\frac{b^{d-1}-(b-a_1)^{d-1}}{(d-1)!\pi(a)}
$$

Therefore, it follows from (\ref{HS.sec.2.eq.1}) that, 
$$
\frac{\partial{L_\lambda}}{\partial{a_1}}=-\frac{b^{d-1}-(b-a_1)^{d-1}}{(d-1)!a_1\pi(a)}+\frac{b^{d-2}+(b-a_1)^{d-2}}{2(d-2)!\pi(a)}+2a_1\lambda\mbox{.}
$$
and that, when $i\geq2$,
$$
\frac{\partial{L_\lambda}}{\partial{a_i}}=-\frac{b^{d-1}-(b-a_1)^{d-1}}{(d-1)!a_i\pi(a)}+\frac{b^{d-2}-(b-a_1)^{d-2}}{2(d-2)!\pi(a)}+2a_i\lambda\mbox{.}
$$

Consider an integer $j$ such that $2\leq{j}\leq{d}$. According to (\ref{HS.sec.2.eq.2}),
$$
a_j\frac{\partial{L_\lambda}}{\partial{a_1}}-a_1\frac{\partial{L_\lambda}}{\partial{a_j}}=0\mbox{.}
$$

Multiplying this equality by $-2a_1a_j/b^{d+1}$, and using the above expressions for the partial derivatives, it can be rewritten as
\begin{equation}\label{HS.sec.3.lem.1.eq.1}
\alpha{x^2}+\beta{x}+\gamma=0\mbox{,}
\end{equation}
where $x=a_j/b$, $y=1-a_1/b$, and
$$
\left\{
\begin{array}{l}
\alpha=2-2y^{d-1}-(d-1)(1-y)(1+y^{d-2})\mbox{,}\\
\beta=(d-1)(1-y)^2(1-y^{d-2})\mbox{,}\\
\gamma=-2(1-y)^2(1-y^{d-1})\mbox{.}\\
\end{array}
\right.
$$

Note that $y$ belongs to the interval $]0,1[$ because $0<a_1<b$. By Proposition~\ref{HS.sec.3.prop.1}, $x$ cannot belong to $[1,+\infty[$ when $d\geq6$, thereby contradicting the assumption that $a_j\geq{b}$. It remains to reach a contradiction when $d=5$. In this case, the two solutions of the quadratic equation (\ref{HS.sec.3.lem.1.eq.1}) are $y+1$ and $(y^2+1)/(y+1)$. Only the former belongs to $[1,+\infty[$, hence
$$
a_1+a_j=2b\mbox{.}
$$

In particular, all the coordinates of $a$ but the first are equal to $2b-a_1$. Now recall that $\|a\|^2=1$ and that (\ref{HS.sec.1.eq.2}) holds. These equalities yield
$$
\left\{
\begin{array}{l}
a_1^2+4(2b-a_1)^2=1\mbox{,}\\
2b-a_1=2t/3\mbox{.}
\end{array}
\right.
$$

Replacing $2b-a_1$ by $2t/3$ in the first of these equalities and remembering the assumption that $t>\sqrt{d-2}/2$, one obtains $a_1^2<-1/3$, a contradiction.
\end{proof}

The main result of the article is the following theorem. 

\begin{thm}\label{HS.sec.3.thm.1}
Assume that $t$ satisfies
$$
\frac{\sqrt{d-2}}{2}<t<\frac{\sqrt{d}}{2}
$$
and consider a hyperplane $H$ at distance $t$ from the center of $[0,1]^d$. If $d\geq5$, then the $(d-1)$-dimensional volume of $H\cap[0,1]^d$ is at most
$$
\frac{d^{d/2}}{(d-1)!}\!\left(\frac{\sqrt{d}}{2}-t\right)^{\!\!d-1}
$$
with equality if and only if $H$ is orthogonal to a diagonal of $[0,1]^d$.
\end{thm}
\begin{proof}
The proof proceeds just as that of Theorem \ref{HS.sec.2.thm.1}, except that Lemmas~\ref{HS.sec.2.lem.3} and \ref{HS.sec.3.lem.1} are used instead of Lemmas \ref{HS.sec.2.lem.1} and \ref{HS.sec.2.lem.2}.
\end{proof}

The remainder of the section is devoted to proving Proposition \ref{HS.sec.3.prop.1}. Recall that, in the statement of this proposition, $y$ is a number in $]0,1[$ and
$$
\left\{
\begin{array}{l}
\alpha=2-2y^{d-1}-(d-1)(1-y)(1+y^{d-2})\mbox{,}\\
\beta=(d-1)(1-y)^2(1-y^{d-2})\mbox{,}\\
\gamma=-2(1-y)^2(1-y^{d-1})\mbox{.}\\
\end{array}
\right.
$$

The proposition states that the solutions of the quadratic equation
\begin{equation}\label{HS.sec.3.eq.1}
\alpha{x^2}+\beta{x}+\gamma=0
\end{equation}
cannot belong to $[1,+\infty[$ when $d\geq6$. 
It suffices to show that $\alpha$, $2\alpha+\beta$, and $\alpha+\beta+\gamma$ are negative. Indeed, the latter inequality implies that the left-hand side of (\ref{HS.sec.3.eq.1}) is negative when $x=1$ and the other two that it is an decreasing function of $x$ in the interval $[1,+\infty[$ (because its derivative is negative). Each of these three inequalities will be proven in a separate proposition.

\begin{prop}\label{HS.sec.3.prop.2}
If $d\geq4$, then $\alpha$ is negative.
\end{prop}
\begin{proof}
First observe that, when $d\geq4$ the second derivative
$$
\frac{\partial^2\alpha}{\partial{y^2}}=-(d-1)(d-2)(d-3)(1-y)y^{d-4}
$$
is negative. As a consequence, the derivative
$$
\frac{\partial\alpha}{\partial{y}}=(d-1)\!\left[1-(d-2)y^{d-3}+(d-3)y^{d-2}\right]
$$
is a strictly decreasing function of $y$ on the interval $]0,1[$. As that derivative is continuous on $]0,1]$ and equal to $0$ when $y=1$, it must be positive when $y$ belongs $[0,1[$. This proves that, in turn $\alpha$ is strictly increasing, as a function of $y$, on the interval $]0,1[$. The proposition follows from the observation that $\alpha$ is a continuous function of $y$ and that is is equal to $0$ when $y=1$.
\end{proof}

\begin{prop}\label{HS.sec.3.prop.3}
If $d\geq6$, then $2\alpha+\beta$ is negative.
\end{prop}
\begin{proof}
One obtains from the expressions of $\alpha$ and $\beta$ that
$$
2\alpha+\beta=4\!\left[1-y^{d-1}\right]-(d-1)(1-y)\!\left[1+y+(3-y)y^{d-2}\right]\!\mbox{.}
$$

As an immediate consequence,
\begin{equation}\label{HS.sec.3.prop.3.eq.1}
\frac{1}{y}\frac{\partial2\alpha+\beta}{\partial{y}}=(d-1)\!\left[2-y^{d-4}(3d-6-4(d-2)y+dy^2)\right]\mbox{.}
\end{equation}

In turn, one obtains
$$
\frac{\partial}{\partial{y}}\!\left(\frac{1}{y}\frac{\partial2\alpha+\beta}{\partial{y}}\right)=(d-1)(d-2)y^{d-5}(1-y)(12-3d+dy)\mbox{.}
$$

If $d\geq6$, that derivative is negative. Since (\ref{HS.sec.3.prop.3.eq.1}) is a continuous function of $y$ on the interval $]0,1]$, and since it vanishes when $y=1$, it must be positive on $]0,1[$. Therefore, $2\alpha+\beta$ is a strictly increasing function of $y$ on the interval $]0,1[$. Since that function is continuous on $]0,1]$ and equal to $0$ when $y=1$, this shows that $2\alpha+\beta$ is negative when $y$ belongs to $]0,1[$. 
\end{proof}

\begin{prop}\label{HS.sec.3.prop.3}
If $d\geq6$, then $\alpha+\beta+\gamma$ is negative.
\end{prop}
\begin{proof}
One obtains from the expressions of $\alpha$, $\beta$, and $\gamma$ that
\begin{equation}\label{HS.sec.3.prop.3.eq.1}
\frac{\alpha+\beta+\gamma}{y}=-(d-5)+(d-3)y+(y-2)y^{d-3}\!\left[(d-1)(1-y)+2y^2\right]\!\mbox{.}
\end{equation}

That ratio is a continuous function of $y$ on $]0,1]$ that is equal to $0$ when $y=1$. As a consequence, it suffices to show that its derivative,
\begin{multline*}
\displaystyle\frac{\partial}{\partial{y}}\!\left(\frac{\alpha+\beta+\gamma}{y}\right)=d-3\\
\hfill+y^{d-4}\!\left[-(d-1)\left(2(d-3)-3(d-2)y+(d+3)y^2\right)+2dy^3\right]\!\mbox{,}
\end{multline*}
is positive when $0<y<1$. Again, observe that this derivative is a continuous function of $y$ on $]0,1]$ and that it vanishes when $y=1$. Therefore, it is sufficient to show that the second derivative of the ratio (\ref{HS.sec.3.prop.3.eq.1}) with respect to $y$ is negative when $0<y<1$. That second derivative is
$$
\displaystyle\frac{\partial^2}{\partial{y^2}}\!\left(\frac{\alpha+\beta+\gamma}{y}\right)=(d-1)(1-y)y^{d-5}(Ay^2+By+C)\mbox{,}
$$
where $A$, $B$, and $C$ are defined as
$$
\left\{
\begin{array}{l}
A=-2d\mbox{,}\\
B=-6-d+d^2\mbox{,}\\
C=-24+14d-2d^2\mbox{.}\\
\end{array}
\right.
$$

Observe that
$$
2A+B=(d+1)(d-6)\mbox{.}
$$

In particular, $2A+B$ is non-negative when $d\geq6$. As in addition $A$ is negative, $2Ay+B$ is necessarily positive when $0<y<1$. This implies that $Ay^2+By+C$ is a strictly increasing function of $y$ on $]0,1[$. Finally,
$$
A+B+C=-(d-5)(d-6)\mbox{.}
$$

That sum is equal to zero when $d=6$ and negative when $d\geq7$. As a consequence, $Ay^2+By+C$ is negative when $0<y<1$, and so is the second derivative of (\ref{HS.sec.3.prop.3.eq.1}) with respect to $y$, as desired.
\end{proof}

\section{Deeper sections of the hypercube}\label{HS.sec.4}

Consider a closed, $d$-dimensional ball of radius $t$ centered at the center of $[0,1]^d$ and such that the vertices of $[0,1]^d$ are outside of $B$. Recall that the interior of $B$ contains the centers of the $k$-dimensional faces of $[0,1]^d$ when
$$
t>\frac{\sqrt{d-k}}{2}\mbox{.}
$$

Consider a hyperplane $H$ of $\mathbb{R}^d$ tangent to $B$ and identify $\mathbb{R}^{d-1}$ with the subspace of $\mathbb{R}^d$ spanned by the first $d-1$ coordinates. The strategy of the proof exposed here in the case when $k=2$ is the following. The first step consists in finding the maximal value for the volume of $H\cap[0,1]^{d-1}$ when the line orthogonal to $H$ through the origin of $\mathbb{R}^d$ is contained in $\mathbb{R}^{d-1}$. This bit of the proof can be thought of as an induction on $k$ because, in this case, $B$ contains the centers of the $(k-1)$-dimensional faces of $[0,1]^{d-1}$ in its interior. By Proposition~\ref{HS.sec.2.prop.1} and the results on the regularity of the $(d-1)$-dimensional volume of $H\cap[0,1]^d$ presented in Section \ref{HS.sec.1}, it is then shown that the largest value for that volume is attained at a critical point of the corresponding Lagrangian function. The necessary conditions given by the Lagrange multipliers theorem reduce the search of that critical point to showing that a certain quadratic equation, whose coefficients are polynomials (of degree about $d$) in the coordinates of a vector orthogonal to $H$, cannot have a solution in the interval $[1,+\infty[$.

Since the regularity results of Section \ref{HS.sec.1} are valid for any $k$ less than $d$, it is tempting to carry on with the inductive process in order to extend the result to greater values of $k$, above a certain dimension. However, two main difficulties arise. The first of these difficulties is combinatorial: the number of ways $H$ cuts the hypercube quickly increases with $k$. For instance, when
$$
\frac{\sqrt{d-3}}{3}<t<\frac{\sqrt{d-2}}{2}\mbox{,}
$$
the vertex set of $[0,1]^d$ can still be split by $H$ in the two ways studied above where either exactly one or exactly two (adjacent) vertices of $[0,1]^d$ are on one of the sides of $H$, but it can be split in three more ways. More precisely, there can also be exactly three or exactly four vertices on one of the sides of $H$. In the former case, these vertices are necessarily three of the four vertices of a square face of $[0,1]^d$. In the latter case, they are either the four vertices of a square face of $[0,1]^d$ or a vertex of $[0,1]^d$ and three of the vertices adjacent to it. Theorem \ref{HS.sec.1.thm.2} provides a different volume formula in each case, and the Lagrange multipliers strategy needs to be applied to each of them.

The second difficulty is analytic. Indeed, locating the solutions of the above mentioned quadratic equation gets much more complicated as the polynomial coefficients of that equation become multivariate and the degree with respect to each variable is about $d$. In particular, generalizing Proposition \ref{HS.sec.3.prop.1} to larger values of $k$ would require bounding linear combinations of these polynomials, which, using a Lagrange multipliers approach would mean solving parametric systems of polynomial equations of degree about $d$ in each variable, in order to find the critical points of the Lagrangian function. Finally, it is shown in \cite{Konig2021,KonigKoldobsky2011} that, when $d\leq{4}$, the $(d-1)$-dimensional volume of $H\cap[0,1]^d$ changes behaviors as $k$ increases and becomes locally minimal when $H$ is orthogonal to a diagonal of $[0,1]^d$ instead of locally maximal. A similar change is likely to happen in higher dimensions, possibly causing further difficulties in the analysis of the problem. In this case, it would be interesting to estimate, for each dimension, the smallest value of $k$ for which that change takes place.

\bibliography{HypercubeSections}
\bibliographystyle{ijmart}

\end{document}